\documentclass{amsart}
\usepackage{amsfonts}
\usepackage{amsmath}
\usepackage{amssymb}
\usepackage{geometry}

\DeclareMathOperator{\lcm}{lcm}
\newtheorem{theorem}{Theorem}[section]
\newtheorem{lemma}[theorem]{Lemma}
\newtheorem{proposition}[theorem]{Proposition}
\newtheorem{corollary}[theorem]{Corollary}
\theoremstyle{definition}
\newtheorem{definition}[theorem]{Definition}

\newtheorem{notation}[theorem]{Notation}
\newtheorem *{Theorem A}{Theorem A}
\newtheorem *{Theorem B}{Theorem B}
\newtheorem *{HS Conjecture}{Herzog-Sch\"{o}nheim Conjecture}

\newcommand\blfootnote[1]{%
	\begingroup
	\renewcommand\thefootnote{}\footnote{#1}%
	\addtocounter{footnote}{-1}%
	\endgroup
}

\begin{document}
\title[(HSC) for small groups and harmonic subgroups]{The Herzog-Sch\"{o}nheim Conjecture for small groups and harmonic subgroups}

\address{Vakgroep Wiskunde, Vrije Universiteit Brussel, Pleinlaan 2, 1050 Brussels, Belgium}
\address{Department of Mathematics, Technion - Israel Institute of Technology, Haifa 32000, Israel}
\author{Leo Margolis}
\email{leo.margolis@vub.ac.be}
\author{Ofir Schnabel}
\email{os2519@yahoo.com}
\date{\today}

\begin{abstract}
We prove that the Herzog-Sch\"{o}nheim Conjecture holds for any group $G$ of order smaller than $1440$.
In other words we show that in any non-trivial coset partition $\{g_i U_i\}_{i=1}^n $  of $G$ there exist distinct $1 \leq i, j \leq n$ such that $[G:U_i]=[G:U_j]$.

We also study interaction between the indices of subgroups having cosets with pairwise trivial intersection and harmonic integers. We prove that if $U_1$,...,$U_n$ are subgroups of $G$ which have pairwise trivially intersecting cosets and $n \leq 4$ then $[G:U_1]$,...,$[G:U_n]$ are harmonic integers.
\end{abstract}
\maketitle

\blfootnote{The first author has been supported by an individual Marie-Curie Individual Fellowship from H2020 EU-project 705112-ZC and the FWO (Research Foundation Flanders). The second author has been supported by the Minerva Stiftung and ISF grant 797/14.}

\section{Introduction}\label{intro}\pagenumbering{arabic} \setcounter{page}{1}
In the 1950's H. Davenport, L. Mirsky, D. Newman and R. Rado proved that if the integers are partitioned  by a finite set of arithmetic progressions, then the largest difference must appear more than once. I.e. if $g_1$,...,$g_n$ and $a_1 \leq a_2 \leq ... \leq a_n$ are integers such that $\{g_i + a_i\mathbb{Z}\}_{i=1}^n$ is a partition of $\mathbb{Z}$ then $a_{n-1} = a_n$. This affirmed a conjecture of P. Erd\H{o}s and opened a broad area of research (see \cite{Porubsky} for a detailed bibliography).
Since an arithmetic progression is a coset of the integers, such partitions can be considered as coset partitions of the group $\mathbb{Z}$.
It was then natural to ask how groups in general can be partitioned by their cosets and if statements similar to the theorem of Davenport, Mirsky, Newman and Rado still hold.

In 1974, M. Herzog and J. Sch\"{o}nheim posed the following.
\begin{HS Conjecture}[HSC] \cite{HerzogSchonheim}
Let $\{g_i U_i\}_{i=1}^n $ be a non-trivial partition of a group $G$ into cosets,
where $U_1, U_2,\ldots ,U_n$ are subgroups of finite index in $G$. Then the
indices $[G:U_1],\ldots, [G:U_n]$ cannot be pairwise distinct.
\end{HS Conjecture}
It is known that in order to prove the Herzog-Sch\"{o}nheim Conjecture it is sufficient to prove it for finite
groups \cite{KorecZnam}. The conjecture holds for finite pyramidal groups (i.e. groups having a Sylow tower)  \cite{BFF}, so in particular for supersolvables groups. Some effort was put into proving the conjecture for groups $G$ whose order satisfies some restriction. E.g. it was proven for a group $G$ in \cite{GinosarSchnabel} if $|G|$ has at most two different prime divisors or exactly three prime divisors and $6$ does not divide $|G|$. It was also proven in \cite{Ginosar} if $|G|$ is smaller than 240.

The main theorem of this note is the following.
\begin{Theorem A}
Any group of order less than $1440$ satisfies the Herzog-Sch\"{o}nheim Conjecture.
\end{Theorem A}

In the proof of Theorem A we will use a group theoretical version of a notion connected to the original problem of Erd\H{o}s.
We say that an $n$-tuple $(a_1$,...,$a_n)$ of positive integers is $\mathbb{Z}$-\textit{harmonic} if there are integers $m_1$,...,$m_n$ such that the arithmetic progressions $m_1 + a_1\mathbb{Z}$,...,$m_n+a_n\mathbb{Z}$ have pairwise trivial intersection. This notion was introduced, simply as harmonic integers, in \cite{HuhnM} and studied e.g. in \cite{Sun92, Chen}. We will use the following generalization of this concept.

\begin{definition} Let $G$ be a group. 
An $n$-tuple of positive integers $(a_1$,...,$a_n)$ is called $G$-\textit{harmonic}
if there exist subgroups $U_1,...,U_n$ of $G$ and $g_1$,...,$g_n\in G$
such that the cosets $g_1U_1,...,g_nU_n$ have pairwise trivial intersection and $[G:U_i]=a_i$  for any $1\leq i\leq n$.
\end{definition}


Harmonic $G$-tuples are connected to coset partitions of groups by the observation that if $\{g_i U_i\}_{i=1}^n $ is a partition of $G$ into cosets then the $n$-tuple of indices $([G:U_1]$,...,$[G:U_n])$ is $G$-harmonic.	
In a similar manner to the correspondence between Erd\H{o}s' original problem and the Herzog-Sch\"{o}nheim Conjecture one can ask for a correspondence between $G$-harmonic and $\mathbb{Z}$-harmonic tuples.\\

\textbf{Question 1:} (Y. Ginosar) Let $G$ be a group and $n$ a positive integer.
Is every $G$-harmonic $n$-tuple also $\mathbb{Z}$-harmonic?\vspace{.2cm}

Notice that if Question 1 admits a positive answer for any group $G$ and any positive integer $n$ then by the theorem of Davenport, Mirsky, Newman and Rado in particular the Herzog-Sch\"{o}nheim Conjecture holds.

Z.-W. Sun asked a related question \cite[Conjecture 1.2]{Sun06}. Namely, for a group $G$, does the existence of a $G$-harmonic $n$-tuple 
$(a_1$,...,$a_n)$ imply that $\gcd(a_i, a_j) \geq n$ for some distinct $1\leq i,j \leq n$? This has been confirmed by W.-J. Zhu for $n \leq 4$ \cite{Zhu} for any group $G$. It has also been proved for $G = \mathbb{Z}$ if $n \leq 20$ \cite{OBryant}.

We obtain the following surprising theorem, which is in fact a generalization of Zhu's result. This result is a key ingredient in the proof of Theorem A.
\begin{Theorem B}
Let $G$ be a group and let $(a_1$,...,$a_n)$ be a $G$-harmonic tuple, where $n \leq 4$. Then $(a_1$,...,$a_n)$ is also $\mathbb{Z}$-harmonic. I.e. for $n \leq 4$, Question 1 admits a positive answer for any group $G$.
\end{Theorem B}

To prove Theorem B we classify all possible $G$-harmonic $n$-tuples for $n\leq 4$. We are not aware of a counterexample to Question 1 for any group $G$.

In Section~\ref{sec_Preliminaries} we recall some preliminaries, explain our strategy and obtain arithmetical restrictions on the indices of subgroups involved in a possible minimal counterexample to the Herzog-Sch\"onheim Conjecture. In Section~\ref{sec_Observs} we obtain general results about coset partitions and $G$-harmonic tuples which are then applied, together with the arithmetical restrictions obtained before, to prove Theorems A and B in Section~\ref{sec_Proofs}. It turns out that the proof of Theorem B boils down to excluding three types of $G$-harmonic
tuples. This is achieved in Propositions~\ref{prop_gcd2}, \ref{prop_4Tuple3Case} and \ref{prop_4Tuple2Case} in which we consider one type of $3$-tuples and two types of $4$-tuples. To prove Theorem A one needs to handle also a $5$-tuple, and this is done in Proposition~\ref{prop_5Tupler11}.

\section{Preliminaries, Egyptian Fractions and Strategy}\label{sec_Preliminaries}
For a group $G$, we will frequently use the notion of harmonic tuple of subgroups hereby explained.
\begin{definition} Let $G$ be a group and $U_1,...,U_n$ subgroups of $G$ of finite index. We call the tuple $(U_1$,...,$U_n)$ \textit{harmonic}, if there exist elements $g_1$,...,$g_n$ in $G$ such that the cosets $g_1U_1,...,g_nU_n$ have pairwise trivial intersection.
\end{definition}
Notice that for a group $G$ and $U_1,...,U_n$ subgroups of $G$ of finite index, $([G:U_1]$,...,$[G:U_n])$ is $G$-harmonic if $(U_1$,...,$U_n)$ is harmonic.
\subsection{Known Results}

Let $G$ be a group and $U_1$,...,$U_n$ subgroups of finite index of $G$. We say that a partition $\{g_i U_i\}_{i=1}^n $ of $G$ admits multiplicity if $[G:U_i] = [G:U_j]$ for some distinct $1 \leq i,j \leq n$.
In this section we recall some conditions concerning coset
partitions without multiplicity obtained in \cite{GinosarSchnabel} and
\cite{Ginosar}. An important one is the so called group theoretical
Chinese reminder theorem.

\begin{lemma}\label{th:CRT1} (see e.g. \cite[Corollary 2.1]{GinosarSchnabel})
Let $U$ and $V$ be two subgroups of a group $G$. If $UV = G$ then $(U,V)$ is not harmonic. In particular if $r$ and $s$ are coprime integers then $(r,s)$ is not $G$-harmonic for any group $G$.
\end{lemma}

The first two statements of the next lemma are clear, the third follows from Lemma~\ref{th:CRT1} and the last one from \cite[Lemma 2.3]{GinosarSchnabel}.
\begin{lemma}\label{lemma_restrindices}
Let $\{g_i U_i\}_{i=1}^n $ be a coset partition of a group $G$ without multiplicity. Set $[G:U_i] = a_i$ for $1 \leq i \leq n$. Then
\begin{itemize}
\item[a)] $a_i \neq a_j$ for distinct $1 \leq i, j \leq n$.
\item[b)] $\sum_{i=1}^n \frac{1}{a_i} = 1$.
\item[c)] $\gcd(a_i, a_j) >1$ for any $1 \leq i, j \leq n$.
\item[d)] If $G$ is a counterexample to (HSC) of minimal order then $a_i > 2$ for any $1 \leq i \leq n$.
\end{itemize}
\end{lemma}

From Lemma~\ref{lemma_restrindices} it follows that a minimal counterexample to (HSC) gives rise to an Egyptian fraction
\begin{align}\label{eq_EgyFr}
1 = \sum_{i=1}^n \frac{1}{a_i},
\end{align}
where $a_1$,...,$a_n$ are pairwise distinct integers bigger than $2$ such that $\gcd(a_i, a_j) > 1$ for any distinct pair $1 \leq i,j \leq n$.

It was shown in \cite{Ginosar} with aid of computers that Lemma~\ref{lemma_restrindices} implies (HSC) for groups of order smaller than $240$. For $|G| = 240$ there exist Egyptian fractions \eqref{eq_EgyFr} satisfying Lemma~\ref{lemma_restrindices}. Using new techniques we improve this bound to 1440 in Theorem A.

\subsection{Arithmetical restrictions for indices}

In this section we show that if $G$ is a group of order less than $1440$ and $\{g_i U_i\}_{i=1}^n$ is a coset partition of $G$ without multiplicity then the greatest common divisor of $([G:U_1],...,[G:U_n])$ is divisible by $2$ or $3$. This will reduce the number of Egyptian fractions \eqref{eq_EgyFr} we will have to consider. Note that $\gcd([G:U_1],...,[G:U_n]) = 1$ is not excluded by Lemma~\ref{th:CRT1}, which makes only a statement about a pair of subgroups.

The following lemmas will help us to reduce the number of
Egyptian fractions~\eqref{eq_EgyFr} we need to consider.

\begin{lemma}\label{lemma_arith3primes}
Let $G$ be a group whose order is divisible by exactly three pairwise different primes.
If there exists a coset partition $\{g_iU_i\}_{i=1}^n$ of $G$ without
multiplicity then there exists some $1 \leq j \leq n$ such that $a_j=[G:U_j]$ is a power of a prime.
\end{lemma}
\begin{proof}
Arguing by contradiction assume that $a_j$ is not a prime power for any $1 \leq j \leq n$. By \cite[Theorem C]{GinosarSchnabel} we may assume that the prime divisors of $|G|$ are $2$, $3$ and $p$ for some prime $p \geq 5$.
Since $a_j$ is not a power of a prime for any $1 \leq j \leq n$ it is divisible by either $6$, $2p$ or $3p$.
Now,
\begin{align*}
\sum_{x=1,z=1}^{\infty}\frac{1}{2^x p^y} &= \left(\sum_{x=1}^\infty \frac{1}{2^x} \right) \left(\sum_{z=1}^\infty \frac{1}{p^z}\right) =  \frac{1}{p-1}\leq \frac{1}{4}, \\
\sum_{y=1, z=1}^{\infty}\frac{1}{3^y p^z} &=  \left(\sum_{y=1}^\infty \frac{1}{3^y} \right) \left(\sum_{z=1}^\infty \frac{1}{p^z}\right) = \frac{1}{2(p-1)}\leq \frac{1}{8},\\
\sum_{x=1, y=1, z=0}^{\infty}\frac{1}{2^x 3^y p^z} &=  \left(\sum_{x=1}^\infty \frac{1}{2^x} \right) \left(\sum_{y=1}^\infty \frac{1}{3^y} \right) \left(\sum_{z=1}^\infty \frac{1}{p^z}\right) =  \frac{1}{2}\left(1+\frac{1}{p-1}\right)\leq \frac{5}{8}.
\end{align*}
Clearly, $\sum_{i=1}^n\frac{1}{a_i}$ is smaller then the sum of all the above which is at most $1$. This contradicts Lemma~\ref{lemma_restrindices}.
\end{proof}


In general arithmetical arguments are not enough any more to show a statement similar to Lemma~\ref{lemma_arith3primes}. But for groups of small order this can still be done explicitly.

\begin{lemma}\label{lemma_CommonPrimeDiv}
Let $G$ be a group of order smaller than $1440$ such that
there exists a coset partition $\{g_iU_i\}_{i=1}^n$ of $G$ without
multiplicity. Then there is $1 \leq j \leq n$ such that $a_j = [G:U_j]$ is a power of a
prime.
\end{lemma}

\begin{proof}
By Lemma~\ref{lemma_arith3primes} we can assume that the order of $G$ is divisible by at least four primes. Furthermore, if the order of $G$ is squarefree then $G$ is pyramidal, by a famous corollary from Burnside's $p$-complement theorem \cite[IV, Satz 2.7]{Huppert}.

For an integer $z$ denote by $d(z)$ the sum of the inverses of all divisors of $z$ which are divisible by at least two primes. So e.g. $d(24) = \frac{1}{6} + \frac{1}{12} + \frac{1}{24}$. Then $d(z) < d(p \cdot z)$ for any prime $p$. Assume $z = p^e \cdot q^f \cdot m$ such that $p$ and $q$ are primes, $m$ is comprime with $pq$ and $e,f \geq 0$. If $p < q$ and $f \geq e$ then $d(z) < p^f \cdot q^e \cdot m$. Considering the prime decompositions of all numbers smaller than $1440$ which have more than three prime divisors and are not squarefree there are two critical cases remaining, namely $840 = 2^3 \cdot 3 \cdot 5  \cdot 7$ and $1260 = 2^2 \cdot 3^2 \cdot 5  \cdot 7$. So
$$d(840) = \frac{737}{840}, \ \ d(1260) = \frac{1171}{1260} $$
implies the lemma.
\end{proof}


We proceed by showing that the prime involved in the previous lemma is in fact at most $3$.
\begin{lemma}\label{lemma_SmallGroupsInds}
Let $G$ be a group such that there exists a coset partition
$\{g_iU_i\}_{i=1}^n$ of $G$ without multiplicity. If $|G|<1440$
then $\gcd([G:U_1],...,[G:U_n])$ is divisible by $2$ or $3$.
\end{lemma}
\begin{proof}
Set $a_i=[G:U_i]$ for $1 \leq i \leq n$. Then by Lemma~\ref{lemma_restrindices} and Lemma~\ref{lemma_CommonPrimeDiv} there exist certain integers $b_1$,...,$b_n$ and a prime $p$ such that $a_i=pb_i$ for any $1\leq i \leq n$. By Lemma~\ref{lemma_restrindices}
$$ 1 = \sum _{i=1}^n \frac{1}{a_i} = \sum _{i=1}^n \frac{1}{pb_i}=\frac{1}{p}\sum _{i=1}^n \frac{1}{b_i}.$$
So $\sum _{i=1}^n \frac{1}{b_i}=p$.
 Now, since $|G|<1440$, the order of $G$ admits
at most $4$ prime factors $p_1$, $p_2$, $p_3$, $p_4$ and therefore
$$\sum _{i=1}^n \frac{1}{b_i}<\prod_{k=1}^4\left(\sum _{j=0}^{\infty}\frac{1}{p_k^j}\right) \leq \prod_{k=1}^4 \left(1 + \frac{1}{p_k-1}\right) \leq   \frac{105}{24}<5.$$
Here in the last step we put $\{p_1,p_2,p_3,p_4\}=\{2,3,5,7\}.$
\end{proof}

\subsection{Strategy} In view of Lemma~\ref{lemma_restrindices} and Lemma~\ref{lemma_SmallGroupsInds} using elementary programming one can easily obtain all possible Egyptian fractions \eqref{eq_EgyFr} which could correspond to a counterexample to the (HSC) of order smaller than 1440. Certain specific subsets of indices appear in all of these Egyptian fractions. The proof of Theorem B consists in showing that a specific form of $3$-tuples (Proposition~\ref{prop_gcd2}) and two specific forms of $4$-tuples (Propositions~\ref{prop_4Tuple3Case}, \ref{prop_4Tuple2Case}) do not appear as indices in a counterexample. This will suffice to prove that there is no counterexample of order smaller than 1440, except possibly of order 1080. Finally in order to exclude counterexamples of order 1080 we have to show that also a certain $5$-tuple does not appear as a subset of indices in a counterexample (Proposition~\ref{prop_5Tupler11}), i.e. that such tuples are not $G$-harmonic for any group $G$.

We start in the next section by collecting results which do not depend on the specific forms of subsets of indices we have to consider to obtain our main results, but which are very useful to prove them.

\section{Observations on indices of harmonic tuples of subgroups}\label{sec_Observs}
We will first collect some results not involving harmonic tuples of subgroups.
$G$ will always denote a finite group. Whenever we will speak of a group we will mean a finite group.

We start with a purely set theoretical observation.

\begin{lemma}\label{lemma_Subsets}
Let $R, S$ and $T$ be proper subsets of a finite set $M$ such that \newline $M = R \cup S = R \cup T = S \cup T$. Then
$$|R \cap S \cap T| = |R| - |M \setminus S| - |M \setminus T|.$$
\end{lemma}
\begin{proof}
This follows immediately from $M \setminus S \subseteq R$ and $M \setminus T \subseteq R$ together with $M \setminus (S \cup T) = \emptyset$.
\end{proof}

We next collect some basic elementary facts which we will use without further mention.

\begin{lemma} Let $U$ and $V$ be subgroups of $G$.
\begin{enumerate}
\item The $\lcm([G:U],[G:V])$ is a divisor of $[G:U \cap V]$.
\item We have an equality $$|UV| = \frac{|U||V|}{|U \cap V|} = \frac{|G|[G:U\cap V]}{[G:U][G:V]}.$$
\item Let $g_u, g_v \in G$. If $g_u U \cap g_v V = \emptyset$ then $g_u^{-1}g_v \notin UV$.
\end{enumerate}
\end{lemma}

This lemma gives rise to the following useful notion.

\begin{notation}
Let $U$ and $V$ be subgroups of a group $G$ of finite index. We will denote by $\alpha(U,V)$ the integer satisfying
$$[G:U \cap V] = \lcm([G:U],[G:V])\alpha(U,V).$$
\end{notation}

We will need the following observations about double cosets.

\begin{lemma}\label{lemma_DoubleCosets}
Let $U, V, W$ be subgroups of $G$ and $x \in G$.
\begin{enumerate}
\item The double coset $UxV$ has cardinality divisible by $\frac{|G|\lcm([G:U],[G:V])}{[G:U][G:V]}$.
\item $VU \cap WU$ consists of double cosets of $(V \cap W)U$.
\end{enumerate}
\end{lemma}
\begin{proof}
The first statement follows from $|UxV| = |U(xVx^{-1})|$.

For the second statement observe that $(V \cap W)U \subseteq VU \cap WU$ and $VU \cap WU$ consists of left cosets of $U$. Let  $vU = wU \subseteq VU \cap WU$ with $v \in V$, $w \in W$. Then $(V \cap W)vU \subseteq VU$. Since $v = wu$ for a certain $u \in U$ we have $(V \cap W)v \subseteq WU$, so $(V \cap W)vU \subseteq VU \cap WU$.
\end{proof}

The following lemma describes a situation we frequently encounter.

\begin{lemma}\label{lemma_main}
Let $U_1,...,U_n,V$ be subgroups of $G$ such that
$$ |U_1V|+...+|U_nV| \geq |G|. $$
If $VU_i \subseteq U_jU_i$ for any $j > i$, then $(U_1,...,U_n,V)$ is not harmonic.
\end{lemma}
\begin{proof}
Assume to the contrary that $g_1U_1,...,g_nU_n, g_vV$ have pairwise trivial intersection.

\textbf{Claim:} $g_iU_iV \cap g_jU_jV = \emptyset$ for distinct $1 \leq i,j \leq n$.\\
Otherwise, assuming $j > i$ we have
\begin{align}\label{eq_inmain}
g_i u_iv = g_ju_jv' \Leftrightarrow g_j^{-1}g_i = u_jv'v^{-1}u_i^{-1}
\end{align}
for certain $u_i \in U_i$, $u_j \in U_j$ and $v,v' \in V$. Now $v'v^{-1}u_i^{-1} \in VU_i$, so $v'v^{-1}u_i^{-1} \in U_jU_i$, by assumption. Hence \eqref{eq_inmain} implies $g_j^{-1}g_i \in U_jU_i$, a contradiction. So the claim follows.

Since
$$|U_1V|+...+|U_nV| = |g_1U_1V|+...+|g_nU_nV| \geq |G|$$
we obtain $|U_1V|+...+|U_nV| = |G|$ and $G = \sqcup_{z=1}^n g_zU_zV$. Therefore $g_v \in g_iU_iV$ for some $1 \leq i \leq n$, but then $g_i^{-1}g_v \in U_iV$, a contradiction.
\end{proof}

In most situations the following special form of the lemma will suffice.

\begin{corollary}\label{corollary_main}
Let $U, V$ and $W$ be subgroups of $G$ such that $|VU| + |WU| \geq |G|$ and $UV \subseteq WV$. Then $(U,V,W)$ is not harmonic.
\end{corollary}

We next collect some results which follow from different assumptions on the indices of intersections of subgroups.

\begin{lemma}\label{lemma_ThreeAlphas1}
Let $U$, $V$ and $W$ be subgroups of $G$ such that
$$[G:U] = ar_u, \ \ [G:V] = ar_v, \ \ [G:W] = ar_w,$$
where $r_u$, $r_v$, $r_w$ are pairwise coprime. If
$$\alpha(U, V) = \alpha(U,W) = \alpha(V,W) = 1$$
then
$$VU = UV = UW = WU = VW = WV.$$
So $UV$ is a subgroup of $G$ and moreover $[G:UV] = a$.
\end{lemma}
\begin{proof}
Note that $[G:U \cap V \cap W]$ is divisible by $ar_u$, $ar_v$ and $ar_w$. Moreover, since $r_u$, $r_v$ and $r_w$ are pairwise coprime, $[G:U \cap V \cap W]$ is divisible by $ar_ur_vr_w$. Let $\alpha$ be the integer satisfying $[G:U \cap V \cap W] = ar_ur_vr_w\alpha$. Then
$$|U(V \cap W)| = \frac{|G|ar_ur_vr_w\alpha}{(ar_u)(ar_vr_w)} = \frac{|G|\alpha}{a} \geq \frac{|G|}{a}.$$
But since $U(V \cap W) \subseteq UV \cap UW$ and $|UV| = |UW| = \frac{|G|}{a}$ it follows that
$$UV = U(V \cap W) = UW.$$
This calculation can be carried out also permuting $U$, $V$ and $W$ and then
$$UV = UW = VW = VU.$$

Finally $[G:UV] = a$ follows immediately from $|UV| = \frac{|G|ar_ur_v}{(ar_u)(ar_v)}$.
\end{proof}

\begin{proposition}\label{prop_main}
Let $U_1$,...,$U_n$ and $V$ be subgroups of $G$. Let $[G:U_i] = a_ir_i$ for $1 \leq i \leq n$ and $[G:V] = ar$ such that $r_1$,...,$r_n$, $r$ are pairwise coprime. Assume that $\sum_{z = 1}^{n} \frac{1}{a_z} \geq 1$ and for any $1 \leq i \leq n$
\begin{itemize}
\item $a_i \mid a$ and $\gcd(\frac{a}{a_i},r_i) = 1$,
\item $\alpha(U_i,V) =1$.
\end{itemize}
Then $(U_1,...,U_n,V)$ is not harmonic.
\end{proposition}
\begin{proof}
Let $1 \leq i,j \leq n$. Our assumptions imply that the $\lcm(a_ir_i,ar) = ar_ir$ and hence also $[G:U_i \cap V] = ar_ir$ and $|U_iV| = \frac{|G|}{a_i}$. Moreover $\lcm(ar_ir, a_jr_j) = ar_ir_jr$, so $[G:U_i \cap U_j \cap V] = ar_ir_jr\alpha$ for some integer $\alpha$. Thus
$$|(U_i \cap V) U_j| = \frac{|G| ar_ir_jr\alpha}{(ar_ir)(a_jr_j)} = \frac{|G|\alpha}{a_j} \leq |VU_j| = \frac{|G|}{a_j}. $$
It follows that $VU_j = (U_i \cap V) U_j \subseteq U_iU_j$. Since
$$\sum_{z=1}^{n} |VU_z| = \sum_{z=1}^{n} \frac{|G|}{a_i} \geq |G|,$$
the tuple $(U_1,...,U_n,V)$ is not harmonic by Lemma~\ref{lemma_main}.
\end{proof}

\begin{corollary}\label{corol_AllAlpha1}
Let $U_1$,...,$U_n$ and $V$ be subgroups of $G$ such that $[G:U_i] = nr_i$ and $[G:V] = nr$, where $r_1$,..., $r_n$, $r$ are pairwise coprime. Assume further that $\alpha(U_i, V) = 1$ for any $1 \leq i \leq n$. Then $(U_1,...,U_n,V)$ is not harmonic.
\end{corollary}

We next study the relations between triples of subgroups and the indices of their intersections.

\begin{lemma}\label{lemma_ThreeTupleCase}
Let $U$, $V$, $W$ be subgroups of a group $G$ such that $(U,V,W)$ is harmonic.
Assume also that
$$[G:U] = ar_u, \ \ [G:V] = ar_v, \ \ [G:W] = ar_w,$$
where $r_u$, $r_v$, $r_w$ are pairwise coprime. Let $b$ be a positive integer such that $\alpha(U, W) = b$.
\begin{itemize}
\item[a)] If  $\alpha(U, V) = 1$ then $b \leq \alpha(V, W) \gcd(r_v,b)$.
\item[b)] If  $\alpha(U, V) = 1$ and $b = \alpha(V, W) \gcd(r_v,b)$ then $|UV| + |VW| < |G|$.
\item[c)] If $\alpha(U, V) = b$ then $[G:U \cap V \cap W]$ is divisible by $abr_ur_vr_w$.
\end{itemize}
\end{lemma}

\begin{proof}
From
$$\lcm([G:U \cap W], [G:V]) = ar_ur_w\lcm(r_v,b)$$
we have
$$[G:U\cap V \cap W] = ar_ur_w\lcm(r_v,b)\alpha$$
for a certain integer $\alpha$. So
\begin{align}\label{eq1}
|W(U \cap V)| = \frac{|G|ar_ur_w\lcm(r_v,b)\alpha}{(ar_w)(ar_ur_v)} = \frac{|G|\lcm(r_v,b)\alpha}{ar_v} \leq |WV| = \frac{|G|\alpha(V,W)}{a}.
\end{align}
Since
$$\frac{\lcm(r_v,b)}{r_v} = \frac{b}{\gcd(r_v,b)}$$
inequality \eqref{eq1} implies $b \leq \alpha(V ,W)\gcd(r_v,b)$. Using again \eqref{eq1} this implies that if $b = \alpha(V,W)\gcd(r_v,b)$ then $WV = W(U \cap V) \subseteq WU$. Since $(U,V,W)$ is harmonic this implies $|UV| + |VW| < |G|$ by Corollary~\ref{corollary_main}.

For c) note that $[G : U \cap V \cap W]$ is divisible by $ar_ur_v\lcm(r_w,b)$ and by $ar_ur_w\lcm(r_v,b)$. Since $r_v$ and $r_w$ are coprime $[G : U \cap V \cap W]$ is hence divisible by $abr_ur_vr_w$.
\end{proof}

\begin{corollary}\label{corollary_OneSmallTwoBig}
Let $U$, $V$ and $W$ be subgroups of $G$ such that
$$[G:U] = ar_u, \ \ [G:V] = ar_v, \ \ [G:W] = ar_w,$$
where $r_u$, $r_v$, $r_w$ are pairwise coprime. Assume that $\alpha(U,V) = 1$ and $\alpha(U,W) = \alpha(V,W)$. If $|UV| + |VW| \geq |G|$ then
 $(U,V,W)$ is not harmonic.
\end{corollary}
\begin{proof}
Assume to the contrary that $(U,V,W)$ is harmonic. Then by Lemma~\ref{lemma_ThreeTupleCase}a) we know that $\gcd(r_v, \alpha(U, W)) = 1$. Hence Lemma~\ref{lemma_ThreeTupleCase}b) implies $|UV| + |VW| < |G|$, contradicting our assumption.
\end{proof}

\section{Specific Indices}\label{sec_Proofs}
In this section we prove restrictions, for a group $G$, on $G$-harmonic $n$-tuples for $n\leq 5$. This will include the proofs of Theorem A and B. Recall that if $U_1$, $U_2$ are subgroups of $G$ such that $(U_1,U_2)$ is harmonic, then $U_1U_2 \neq G$ by Lemma~\ref{th:CRT1}. We will use this several times without further mention.

\begin{lemma}\label{lemma_gcd2}
Let $U_1$, $U_2$ be subgroups of $G$ such that $(U_1,U_2)$ is harmonic. Assume that $|G:U_i|=2r_i$ for $1 \leq i \leq 2$, where $r_1$, $r_2$ are coprime. Then $\alpha(U_1,U_2) = 1$. Moreover $|U_1U_2| = \frac{|G|}{2}$ and $[G:(U_1 \cap U_2)] = 2r_1r_2$.
\end{lemma}

\begin{proof} Since $(U_1,U_2)$ is harmonic we have $U_1U_2 \neq G$. Note that
$$\lcm([G:U_1],[G:U_2]) = 2r_1r_2,$$
thus $[G:(U_1 \cap U_2)] = 2r_1r_2\alpha(U_1,U_2)$. Since
\[ |U_1U_2| =  \frac{|G|[G:U_1 \cap U_2]}{[G:U_1][G:U_2]} = \frac{|G|2r_1r_2\alpha(U_1,U_2)}{4r_1r_2} \]
and $|U_1U_2| < |G|$ we obtain $\alpha(U_1,U_2) = 1$ and the result follows.\\
\end{proof}

The following is now an immediate application of the preceding lemma and Corollary~\ref{corol_AllAlpha1}. It was also proved by Zhu \cite[Theorem 2.2]{Zhu}.

\begin{proposition}\label{prop_gcd2}
Let $r_1$, $r_2$, $r_3$ be pairwise coprime integers.
Then for any group $G$, the $3$-tuple $(2r_1,2r_2,2r_3)$ is not $G$-harmonic. 
\end{proposition}
\begin{proof}
Assume to the contrary that there exist $U_1$, $U_2$, $U_3$ subgroups of $G$ such that $(U_1,U_2,U_3)$ is harmonic where $[G:U_i]=2r_i$ for $1\leq i\leq 3$. Then by Lemma~\ref{lemma_gcd2} we have $\alpha(U_i, U_j) = 1$ for any $1 \leq i, j \leq 3$. But then Corollary~\ref{corol_AllAlpha1} implies that $(U_1,U_2,U_3)$ is not harmonic.
\end{proof}

The following has also been proven by Zhu \cite[Theorem 3.1]{Zhu}.
\begin{proposition}\label{prop_4Tuple3Case}
Let $r_1$, $r_2$, $r_3$, $r_4$ be pairwise coprime integers.
Then for any group $G$, the $4$-tuple $(3r_1,3r_2,3r_3,3r_4)$ is not $G$-harmonic. 
\end{proposition}
\begin{proof}
Assume to the contrary that there exist $U_1$, $U_2$, $U_3$, $U_4$ subgroups of $G$ such that $(U_1,U_2,U_3,U_4)$ is harmonic where $[G:U_i]=3r_i$ for $1\leq i\leq 4$. Set $\alpha(i,j) = \alpha(U_i,U_j)$ for any $1 \leq i,j \leq 4$. Then
$$|U_iU_j| = \frac{|G|\alpha(i,j)}{3}$$
and so $\alpha(i,j) \leq 2$. Let $\{i,j,k,\ell\} = \{1,2,3,4\}$. We first observe

\textbf{Claim 1:} Not exactly two of $\alpha(i,j), \alpha(i,k), \alpha(j,k)$ equal $2$. \\
Assume w.l.o.g. $\alpha(i,j) = 1$ and $\alpha(i,k) = \alpha(j,k) = 2$. Then
$$|U_iU_j| + |U_jU_k| = \frac{|G|}{3} + \frac{2|G|}{3} = |G|$$
and therefore Corollary~\ref{corollary_OneSmallTwoBig} implies that $(U_i,U_j,U_k)$ is not harmonic.

\textbf{Claim 2:} Not exactly one of $\alpha(i,j), \alpha(i,k), \alpha(j,k)$ equals $2$.\\
Assume w.l.o.g. that $\alpha(i,j)=2$ while $\alpha(i,k) = \alpha(j,k) = 1$. Then Lemma~\ref{lemma_ThreeTupleCase}a) implies
$$2 = \alpha(i,j) \leq \alpha(i,k) \gcd(r_k,2) = \gcd(r_k,2).$$
Hence $r_k$ is even. We already know that $\alpha(i,\ell) = \alpha(j,\ell)$ by Claim 1. If $\alpha(i,\ell) = 1$, then as before $r_\ell$ must be even, contradicting that $r_k$ and $r_\ell$ are coprime. So $\alpha(i,\ell) = \alpha(j,\ell) = 2$. Then by Claim 1 we have $\alpha(k,\ell) = 1$, since otherwise exactly two of $\alpha(j,k)$, $\alpha(j,\ell)$, $\alpha(k,\ell)$ would equal $2$. Thus $\alpha(i,k) = \alpha(j,k) = \alpha(\ell,k) = 1$ and $(U_1,...,U_4)$ is not harmonic by Proposition~\ref{prop_main} (setting $U_k = V$). This proves the claim.

If $\alpha(1,2) = \alpha(1,3) =  \alpha(1,4) = 1$ we can apply Proposition~\ref{prop_main}, so one of them must equal $2$ and by the above $\alpha(x,y) = 2$ for all $1 \leq x,y, \leq 4$.

In this case we will consider the intersections of $U_1U_2$, $U_1U_3$, $U_1U_4$. By Lemma~\ref{lemma_ThreeTupleCase}c) we know that
$$[G:U_i \cap U_j \cap U_k] = 6r_ir_jr_k\alpha$$
for some integer $\alpha$. So by Lemma~\ref{lemma_DoubleCosets} $|U_iU_j \cap U_iU_k|$ is a multiple of
$$|U_i(U_j \cap U_k)| = \frac{|G|\alpha}{3},$$
so in particular of $\frac{|G|}{3}$. If $U_iU_j = U_iU_k$ then by Corollary~\ref{corollary_main} the tuple $(U_i,U_j,U_k)$ is not harmonic. Hence $\alpha = 1$ and $|U_iU_j \cap U_iU_k| = \frac{|G|}{3}$.

Since $|U_iU_j| = |U_iU_k| = \frac{2|G|}{3}$ we obtain $G = U_iU_j \cup U_iU_k$ for order reasons, i.e. because $U_iU_j$ and $U_iU_k$ are big while their intersection is small. So considering $U_1U_2$, $U_1U_3$ and $U_1U_4$ as subsets of $G$ and using Lemma~\ref{lemma_Subsets} we obtain
$$|U_1U_2 \cap U_1U_3 \cap U_1U_4| = |U_1U_2| - |G \setminus U_1U_3| - |G \setminus U_1U_4| = \frac{2|G|}{3} - \frac{|G|}{3} - \frac{|G|}{3} = 0.$$
A contradiction, since clearly this intersection contains $U_1$.
\end{proof}

We continue with the last case relevant for 4-tuples of subgroups. We will use the following lemma to handle it.

\begin{lemma}\label{lemma_244}
Let $U_1$, $U_2$, $U_3$ be subgroups of $G$ such that $[G:U_1] = 2r_1$ and $[G:U_i] = 4r_i$ for $2 \leq i \leq 3$, where $r_1$, $r_2$, $r_3$ are pairwise coprime and $r_1$ is odd. Assume also that $(U_1,U_2,U_3)$ is harmonic. Then
\begin{itemize}
\item $\alpha(U_2, U_3) \neq 2$.
\item If $\alpha(U_2, U_3) = 3$ then $3$ divides $r_1$ and $|U_2(U_1 \cap U_3)| = |U_2U_1 \cap U_2U_3| = \frac{|G|}{4}$.
\end{itemize}
\end{lemma}
\begin{proof}
By Lemma~\ref{lemma_gcd2} we know $|U_1U_2| = |U_1U_3| = \frac{|G|}{2}$.

First assume $\alpha(U_2, U_3) = 2$. This implies $|U_2U_3| = \frac{|G|}{2}$. Moreover
$$[G : U_1 \cap (U_2 \cap U_3)] = 8r_1r_2r_3\alpha$$
for some integer $\alpha$. Then
$$|U_1(U_2 \cap U_3)| = \frac{|G|8r_1r_2r_3\alpha}{(2r_1)(8r_2r_3)} = \frac{|G|\alpha}{2} \leq \frac{|G|}{2}.$$
So $U_1U_2 = U_1U_3$ and hence $(U_1,U_2,U_3)$ is not harmonic by Corollary~\ref{corollary_main}. Thus $\alpha(U_2, U_3) \neq 2$.	

Next assume $\alpha(U_2, U_3) = 3$, so $|U_2U_3| = \frac{3|G|}{4}$. Then Lemma~\ref{lemma_ThreeTupleCase}a) implies that $3$ divides $r_1$. Moreover, setting $[G:U_1 \cap U_2 \cap U_3] = 4r_1r_2r_3\alpha$ for some integer $\alpha$ we have
$$|U_2(U_1 \cap U_3)| = \frac{|G|4r_1r_2r_3\alpha}{(4r_2)(4r_1r_3)} = \frac{|G|\alpha}{4} \leq |U_2U_1| = \frac{|G|}{2}.$$
If $U_2U_1 \subseteq U_2U_3$ then $(U_1,U_2,U_3)$ is not harmonic by Corollary~\ref{corollary_main}. So $\alpha = 1$ and since the cardinality of $U_2U_1 \cap U_2U_3$ is then a multiple of $\frac{|G|}{4}$ by Lemma~\ref{lemma_DoubleCosets} we get $|U_2U_1 \cap U_2U_3| = \frac{|G|}{4}$.
\end{proof}

\begin{proposition}\label{prop_4Tuple2Case}
Let $r_1$, $r_2$, $r_3$, $r_4$ be pairwise coprime integers and assume also that $r_1$ is odd.
Then, for any group $G$, the $4$-tuple $(2r_1,4r_2,4r_3,4r_4)$ is not $G$-harmonic.
\end{proposition}
\begin{proof}
Assume to the contrary that there exist $U_1$, $U_2$, $U_3$, $U_4$ subgroups of $G$ such that 
$(U_1,U_2,U_3,U_4)$ is harmonic where $[G:U_1] = 2r_1$ and $[G:U_i] = 4r_i$ for $2 \leq i \leq 4$.
Hence, there exist $g_1,g_2,g_3,g_4\in G$ such that $g_1U_1$,...,$g_4U_4$ are pairwise trivially intersecting cosets. Assume w.l.o.g. that $g_1 = 1$, in particular $g_i \notin U_1U_i$ for $i \geq 2$. Set $ \alpha(i,j) = \alpha(U_i, U_j)$. From here on let $2 \leq i,j \leq 4$ be different.

By Lemma~\ref{lemma_244} it follows that $\alpha(1,i) = 1$ and $\alpha(i,j) \in \{1,3\}$. By Proposition~\ref{prop_main} not all of the $\alpha(i,j)$ equal $1$ and if, say, $\alpha(2,3) = 3$ also one of $\alpha(2,4)$, $\alpha(3,4)$ equals $3$, say $\alpha(2,4)$. If $\alpha(3,4) = 1$ then, computing as before, we obtain that $|U_3(U_2 \cap U_4)|$ is a multiple of $\frac{|G|}{4}$. This implies that $U_3U_4 \subseteq U_3U_2$ and then $(U_2,U_3,U_4)$ is not harmonic by Corollary~\ref{corollary_main}. So also $\alpha(3,4) = 3$. Summarizing we have
$$\alpha(1,2) = \alpha(1,3) = \alpha(1,4) = 1 \ \ \text{and} \ \ \alpha(2,3) = \alpha(2,4) = \alpha(3,4) = 3.$$

Then by Lemma~\ref{lemma_244} we know that $3$ divides $r_1$. So $[G:U_2 \cap U_3 \cap U_4] = 12r_2r_3r_4\alpha$ for a certain integer $\alpha$. Hence $|U_2U_3 \cap U_2U_4|$ is a multiple of $\frac{|G|\alpha}{4}$ by Lemma~\ref{lemma_DoubleCosets}. For order reasons, i.e. because $U_2U_3$ and $U_2U_4$ are big subsets of size $\frac{3|G|}{4}$ inside $G$, we get
$$|U_2U_3 \cap U_2U_4| \geq \frac{|G|}{2}.$$
But if $|U_2U_3 \cap U_2U_4| = \frac{3|G|}{4}$ then $U_2U_3 = U_2U_4$ and $(U_2,U_3,U_4)$ is not harmonic by Corollary~\ref{corollary_main}. So $|U_2U_3 \cap U_2U_4| = \frac{|G|}{2}$ and $G = U_2U_3 \cup U_2U_4$, for order reasons. The same type of arguments shows that $U_2U_1 \cup U_2U_3 = U_2U_1 \cup U_2U_4 = G$.
So applying Lemma~\ref{lemma_Subsets} we obtain
$$|U_2U_1 \cap U_2U_3 \cap U_2U_4| = |U_2U_1| - |G \setminus U_2U_3| - |G \setminus U_2U_4| = 0.$$
A contradiction, since this intersection contains $U_2$.
\end{proof}

This allows us to prove Theorem B.

\begin{center}
\textbf{Proof of Theorem B}
\end{center}
Let $U_1$, $U_2$, $U_3$, $U_4$ be subgroups of $G$ and set $[G:U_i] = a_i$ for $1 \leq i \leq 4$.  If $(a_1,a_2)$ is not $\mathbb{Z}$-harmonic then $a_1$ and $a_2$ are coprime. So $(U_1,U_2)$ is not $G$-harmonic by Lemma~\ref{th:CRT1}. Next assume that $(a_1,a_2,a_3)$ is not $\mathbb{Z}$-harmonic such that for distinct $1\leq i,j\leq 3$, any $(a_i,a_j)$ is $\mathbb{Z}$-harmonic. Sun's positive answer to Problem 2 from \cite{HuhnM} for 3-tuples (see \cite{Sun92} or the English-language review of this article on MathSciNet) implies that no pair from $\{a_1,a_2,a_3\}$ is coprime and no pair has greatest common divisor bigger than $2$. So there are pairwise coprime integers $r_1$, $r_2$ and $r_3$ such that $a_i = 2r_i$ for $1\leq i \leq 3$. Hence $(U_1,U_2,U_3)$ is not harmonic by Proposition~\ref{prop_gcd2}.

Finally assume that the $4$-tuple $(a_1,a_2,a_3,a_4)$ is not $\mathbb{Z}$-harmonic such that any proper $2$- or $3$-subtuple is $\mathbb{Z}$-harmonic. Then clearly $b = \gcd(a_1,a_2,a_3,a_4) < 4$. In \cite{Sun92} also a positive answer to Problem 1 from \cite{HuhnM} is provided for 4-tuples. This implies that $b = 1$ is not possible and that there are pairwise coprime integers $r_1$, $r_2$, $r_3$ and $r_4$ such that either
\begin{itemize}
\item[i)] $b = 2$, $a_1 = 2r_1$, $a_i = 4r_i$ for $2 \leq i \leq 4$ and $r_1$ is odd or
\item[ii)] $b = 3$ and $a_i = 3r_i$ for $1 \leq i \leq 4$.
\end{itemize}
In case i) the $4$-tuple $(U_1,U_2,U_3,U_4)$ is not harmonic by Proposition~\ref{prop_4Tuple2Case} and in case ii) by Proposition~\ref{prop_4Tuple3Case}. \hfil \qed  \\

To prove also Theorem A we will also need to handle a case involving five subgroups. The calculations here are more involved.

\begin{lemma}\label{lemma_5TupleCase}
Let $U_1$, $U_2$, $U_3$, $U_4$, $U_5$ be subgroups of $G$ such that $[G:U_i] = 3r_i$ for $1 \leq i \leq 2$ and $[G:U_j] = 6r_j$ for $3 \leq j \leq 5$, where $r_1,...,r_5$ are pairwise coprime and $r_1$ and $r_2$ are odd.
If the $5$-tuple $(U_1,U_2,U_3,U_4,U_5)$ is harmonic,
 then up to permutations between $1$ and $2$ and permutations between $3$, $4$ and $5$ the following holds.
\begin{align*}
\alpha(1,3) &= \alpha(2,3) = \alpha(3,5) = \alpha(4,5) = 1, \\
 \alpha(1,2) &= \alpha(1,4) = \alpha(2,4) = \alpha(1,5) = \alpha(2,5) = 2, \\
  \alpha(3,4) &= 3, \\
  3 \mid r_2& \ \  \text{and} \ \ 2 \mid r_3.
\end{align*}
\end{lemma}
\begin{proof}
As before assume $g_1U_1$,..,$g_5U_5$ are pairwise trivially intersecting cosets and assume $g_1 = 1$. Set $\alpha(x,y) = \alpha(U_x,U_y)$ for $1 \leq x,y \leq 5$. Note that $\alpha(x,y) \leq 2$ if $1 \leq x \leq 2$ or $1 \leq y \leq 2$ and $\alpha(x,y) \leq 5$ for any $1 \leq x,y \leq 5$.

To prove the lemma we will first show the following claims. \vspace{.1cm}

\textbf{Claim 1:} $\alpha(1, 2) = 2$ and $\alpha(1, j) = \alpha(2, j)$ for $3 \leq j \leq 5$. \vspace{.1cm}

\textbf{Claim 2:} Let $1 \leq i \leq 2$ and $3 \leq j,k \leq 5$. If $\alpha(i,j) = \alpha(i, k) = 2$ then $\alpha(j,k) = 1$ and $[G:U_i \cap U_j \cap U_k] = 12r_ir_jr_k$. \vspace{.1cm}

\textbf{Claim 3:} Let $1 \leq i \leq 2$ and $3 \leq j,k \leq 5$. If $\alpha(i,j) = \alpha(i,k) = 1$ then $\alpha(j,k) = 1$ and $[G:U_i \cap U_j \cap U_k] = 6r_ir_jr_k$. \vspace{.1cm}

\textbf{Claim 4:} Let $1 \leq i \leq 2$ and $3 \leq j,k \leq 5$. If $\alpha(i,j) = 1$ and $\alpha(i,k) = 2$ then $\alpha(j,k) \in \{1,3\}$. If $\alpha(j,k) = 1$ then $[G:U_i \cap U_j \cap U_k] = 6r_ir_jr_k$ and $2 \mid r_j$. On the other hand, if $\alpha(j,k) = 3$ then $2 \mid r_j$ or $3 \mid r_i$. \vspace{.1cm}

After proving these claims, up to permutations of indices which behave symmetrically, we are left with exactly eight cases to consider for the tuple $(\alpha(x,y))_{1 \leq x,y \leq 5 }$. We will then exclude seven of these cases and finally prove the lemma. \vspace{.1cm}

\textit{Proof of Claim 1:}
Assume first that $\alpha(1,2) = 1$. Fix some $j$ such that $3 \leq j \leq 5$. If $\alpha(1,j) = \alpha(2,j) = 2$ then $|U_1U_2| + |U_1U_j| = |G|$. So we obtain a contradiction by Corollary~\ref{corollary_OneSmallTwoBig}. If $\alpha(1, j) = 1$ and $\alpha(2,j) = 2$ then 	by Lemma~\ref{lemma_ThreeTupleCase}a) the number $r_2$ must be even, contradicting our assumption. So $\alpha(1,j) = \alpha(2,j) = 1$. Hence by Lemma~\ref{lemma_ThreeAlphas1} the set $V = U_1U_2 = U_1U_j = U_jU_1$ is a subgroup of $G$, equal for any choice of $j$ and $[G:V] = 3$.

Then $g_2, g_3, g_4, g_5 \notin V$. Let $g \in G$ such that $g_2 \in gV$. Then $g_3,g_4,g_5 \notin gV$. Hence there is $h \in G$ such that $g_3,g_4,g_5 \in hV$. So $U_1 \cap U_3$, $U_1 \cap U_4$ and $U_1 \cap U_5$ are subgroups of $U_1$ with pairwise trivially intersecting cosets $h^{-1}g_3(U_1 \cap U_3)$, $h^{-1}g_4(U_1 \cap U_4)$ and $h^{-1}g_5(U_1 \cap U_5)$. But since $[U_1:U_1 \cap U_i] = 2r_i$ for $3 \leq i \leq 5$ this contradicts Proposition~\ref{prop_gcd2}. So $\alpha(1,2) = 2$.

Next assume $\alpha(1,j) \neq \alpha(2,j)$, say $\alpha(1,j) = 2$ and $\alpha(2,j) = 1$. Then
$$|U_2U_j| + |U_1U_j| = \frac{|G|}{3} + \frac{2|G|}{3} = |G|.$$
Hence $(U_1,U_2,U_j)$ is not harmonic by Corollary~\ref{corollary_OneSmallTwoBig}. Consequently $\alpha(1,j) = \alpha(2,j)$. This finishes the proof of Claim 1. \vspace{.1cm}

\textit{Proof of Claim 2:}
Considering the three different pairs from $U_i$, $U_j$ and $U_k$ we obtain that $[G:U_i \cap U_j \cap U_k]$ is divisible by $6r_ir_j\lcm(2,r_k)$, $6r_ir_k\lcm(2,r_j)$ and $6r_jr_k\lcm(r_i, \alpha(j,k))$. Hence there is some integer $\alpha$ such that
\begin{align}\label{eq_Claim2}
[G:U_i \cap U_j \cap U_k] = \left\{\begin{array}{ll}12r_jr_k\lcm(r_i, \alpha(j,k)_{2'})\alpha, \ \ \text{if} \ \alpha(j,k) \neq 4 \\ 24r_jr_k\lcm(r_i, \alpha(j,k)_{2'})\alpha, \ \ \text{if} \ \alpha(j,k) = 4. \end{array}\right.
\end{align}

Here $z_{2'}$ denotes the biggest odd divisor of an integer $z$. Then by computing $|U_j(U_i \cap U_k)|$ in all possible cases for $\alpha(j,k)$ we get
$$|U_j(U_i \cap U_k)| = \left\{\begin{array}{lllllll} \vspace*{.1cm}
\frac{5|G|\alpha}{6}, \ \ \text{if} \ \alpha(j,k) = 5, \ 5 \nmid r_i \\ \vspace*{.1cm}
\frac{|G|\alpha}{6}, \ \ \text{if} \ \alpha(j,k) = 5, \ 5 \mid r_i \\ \vspace*{.1cm}
\frac{|G|\alpha}{3}, \ \ \text{if} \ \alpha(j,k) = 4 \\ \vspace*{.1cm}
 \frac{|G|\alpha}{2}, \ \ \text{if} \ \alpha(j,k) = 3, \ 3 \nmid r_i \\ \vspace*{.1cm}
\frac{|G|\alpha}{6}, \ \ \text{if} \ \alpha(j,k) = 3, \ 3 \mid r_i \\ \vspace*{.1cm}
\frac{|G|\alpha}{6}, \ \ \text{if} \ \alpha(j,k) = 2 \\ \vspace*{.1cm}
\frac{|G|\alpha}{6}, \ \ \text{if} \ \alpha(j,k) = 1
  \end{array}\right. $$

Let $1 \leq i' \leq 2$ such that $i \neq i'$. We will consider the different cases for $\alpha(j,k)$ in decreasing order.

Clearly $|U_j(U_i \cap U_k)| \leq |U_jU_i| = \frac{2|G|}{3}$. Hence $\alpha(j,k) = 5$ is not possible, since if $5 \nmid r_i$ then $|U_j(U_i \cap U_k)| = \frac{5|G|}{6}$ and otherwise $5 \nmid r_{i'}$ and $|U_j(U_{i'} \cap U_k)| = \frac{5|G|}{6}$.

Next if $\alpha(j,k) = 4$ and $\alpha = 2$ then $U_jU_i = U_jU_k$ and $|U_jU_i| + |U_jU_k| = \frac{2|G|}{3} + \frac{2|G|}{3}$, contradicting Corollary~\ref{corollary_main}. So in this case $\alpha = 1$.

If $\alpha(j,k) = 3$ and $3 \nmid r_i$ then $U_jU_k \subseteq U_jU_i$ and we can also apply Corollary~\ref{corollary_main} to obtain a contradiction, since $|U_jU_k| + |U_iU_k| = \frac{|G|}{2} + \frac{2|G|}{3}$. But if $3 \mid r_i$ then $3 \nmid r_{i'}$ and we can argue the same way replacing $i$ by $i'$.

If $\alpha(j,k) = 2$ then $\alpha \leq 2$ for order reasons and if $\alpha = 2$ we have again $U_jU_k \subseteq U_jU_i$ and $|U_jU_k| + |U_iU_k| = |G|$, contradicting Corollary~\ref{corollary_main}. So also in this case $\alpha = 1$.

If $\alpha(j,k) = 1$ then also $\alpha = 1$.

So we have $\alpha(j,k) \in \{1,2,4\}$ and $\alpha = 1$. Assume $\alpha(j,k) > 1$. Then
$$|U_i(U_j \cap U_k)| = \frac{|G|}{3}.$$
Notice that also $|U_i(U_{i'} \cap U_j)|$ and $|U_i(U_{i'} \cap U_k)|$ are multiples of $\frac{|G|}{3}$.
If $|U_iU_j \cap U_iU_k|$ or $|U_iU_{i'} \cap U_iU_j|$ or $|U_iU_{i'} \cap U_iU_k|$ is a proper multiple of $\frac{|G|}{3}$ it is equal to $\frac{2|G|}{3}$ by Lemma~\ref{lemma_DoubleCosets}. Then we obtain a contradiction by Corollary~\ref{corollary_main}. On the other hand if all these cardinalities are exactly $\frac{|G|}{3}$ then the union of two of the product sets is the whole group for order reasons. Thus we get by Lemma~\ref{lemma_Subsets} 
$$|U_iU_{i'} \cap U_iU_j \cap U_iU_k| = 0,$$
also a contradiction, since this intersection contains $U_i$.

So we have $\alpha(j,k) = 1$ and $\alpha = 1$. Using \eqref{eq_Claim2} Claim 2 follows. \vspace{.1cm}

\textit{Proof of Claim 3:}
Arguing in a similar way as above we get
$$[G:U_i \cap U_j \cap U_k] = 6r_jr_k\lcm(r_i, \alpha(j,k)).$$
So
$$|U_j(U_i \cap U_k)| = \left\{\begin{array}{lllllll} \vspace*{.1cm}
 \frac{5|G|\alpha}{6}, \ \ \alpha(j,k) = 5, \ 5 \nmid r_i \\ \vspace*{.1cm}
\frac{|G|\alpha}{6}, \ \ \text{if} \ \alpha(j,k) = 5, \ 5 \mid r_i \\ \vspace*{.1cm}
\frac{2|G|\alpha}{3}, \ \ \text{if} \ \alpha(j,k) = 4 \\ \vspace*{.1cm}
 \frac{|G|\alpha}{2}, \ \ \text{if} \ \alpha(j,k) = 3, \ 3 \nmid r_i \\ \vspace*{.1cm}
\frac{|G|\alpha}{6}, \ \ \text{if} \ \alpha(j,k) = 3, \ 3 \mid r_i \\ \vspace*{.1cm}
\frac{|G|\alpha}{3}, \ \ \text{if} \ \alpha(j,k) = 2 \\ \vspace*{.1cm}
\frac{|G|\alpha}{6}, \ \ \text{if} \ \alpha(j,k) = 1
  \end{array}\right. $$
As before let $1 \leq i' \leq 2$ such that $i \neq i'$. Again we will consider the different possibilities for $\alpha(j,k)$ in decreasing order.

As $|U_j(U_i \cap U_k)| \leq |U_jU_i| = \frac{|G|}{3}$ we know that $\alpha(j,k) = 5$ is not possible, using that $5 \mid r_i$ implies $5 \nmid r_{i'}$.

Next $\alpha(j,k) = 4$ and $\alpha(j,k) = 3$ can be excluded for similar reasons. Where in case $\alpha(j,k) =3$ we use $3 \nmid r_{i'}$ if $3 \mid r_i$.

Now assume $\alpha(j,k) = 2$. Then $\alpha = 1$ follows and we have $U_jU_i = U_jU_k$. Also we obtain $|U_i(U_j\cap U_k)| = \frac{|G|}{3}$, so $U_iU_j = U_iU_k$. Considering other combinations of the indices we get
$$U_jU_i = U_jU_k = U_iU_k = U_iU_j =: V, $$
which is a subgroup satisfying $[G:V] = 3$. But since also $U_jU_k = U_{i'}U_k$ the subgroup $V$ contains $U_i$ and $U_{i'}$, so also $U_1U_2$, contradicting $|U_1U_2| = \frac{2|G|}{3}$.

Finally in case $\alpha(j,k) = 1$ we get $\alpha = 1$ for order reasons. This finishes the proof of Claim 3. \vspace{.1cm}

\textit{Proof of Claim 4:}
 Considering again the possible pairings of $U_i$, $U_j$ and $U_k$ we obtain that there are integers $\alpha$ and $\alpha'$ such that
 $$[G:U_i\cap U_j \cap U_k] = 6r_ir_k\lcm(r_j, 2)\alpha = 6r_jr_k\lcm(r_i, \alpha(3,4))\alpha'.$$
 So we obtain (listing only part of the full information)
$$|U_k(U_i \cap U_j)| = \left\{\begin{array}{lllllll} \vspace*{.1cm}
 \frac{5|G|\alpha'}{6}, \ \ \text{if} \ \alpha(j,k) = 5, \ 5 \nmid r_i \\ \vspace*{.1cm}
\frac{2|G|\alpha'}{3}, \ \ \text{if} \ \alpha(j,k) = 4 \\ \vspace*{.1cm}
 \frac{|G|\alpha}{3} = \frac{|G|\alpha'}{2}, \ \ \text{if} \ \alpha(j,k) = 3, \ 3 \nmid r_i, \ 2 \nmid r_j \\ \vspace*{.1cm}
\frac{|G|\alpha'}{3}, \ \ \text{if} \ \alpha(j,k) = 2 \\ \vspace*{.1cm}
\frac{|G|\alpha}{3}, \ \ \text{if} \ \alpha(j,k) = 1, \ 2 \nmid r_j \\ \vspace*{.1cm}
\frac{|G|\alpha}{6}, \ \ \text{if} \ \alpha(j,k) = 1, \ 2 \mid r_j
  \end{array}\right. $$
Again let $1 \leq i' \leq 2$ such that $i \neq i'$ and consider the different possibilities for $\alpha(j,k)$.

If $\alpha(j,k) = 5$ we get a contradiction as before.

If $\alpha(j,k) = 4$ then we get $U_jU_i = U_jU_k$ and this can not happen by Corollary~\ref{corollary_main}, since $|U_jU_i| + |U_iU_k| = \frac{|G|}{3} + \frac{2|G|}{3}$.

Next $\alpha(j,k) = 3$ while $3 \nmid r_i$ and $2 \nmid r_j$ clearly provides a contradiction.

Let $\alpha(j,k) = 2$. Then we compute $|(U_j \cap U_2)U_1| = \frac{|G|}{3}$ and $|(U_j \cap U_k)U_i| = \frac{|G|}{3}$.
So $U_jU_1 \subseteq U_2U_1$, $U_jU_1 \subseteq U_kU_1$, $U_jU_2 \subseteq U_kU_2$ and
$$|U_1U_j| + |U_2U_j| + |U_kU_j| = \frac{|G|}{3} + \frac{|G|}{3} + \frac{|G|}{3} = |G|,$$
 contradicting Lemma~\ref{lemma_main}.

Finally $\alpha(j,k) = 1$ implies $2 \mid r_j$. This completes the proof of Claim 4. \vspace{.1cm}

Set
$$v = (\alpha(1,3), \alpha(1,4), \alpha(1,5), \alpha(3,4), \alpha(3,5), \alpha(4,5)).$$
Combining all previous claims leaves us, up to permutations in the indices $3$, $4$ and $5$, with the following eight possibilities for $v$:
\begin{align*}
v \in \{&(1,1,1,1,1,1), (1,1,2,1,1,1), (1,1,2,1,3,1), (1,1,2,1,3,3), \\
 &(1,2,2,1,1,1), (1,2,2,3,1,1), (1,2,2,3,3,1), (2,2,2,1,1,1)\}.
\end{align*}
First observe that $v = (1,1,1,1,1,1)$ is not possible by Proposition~\ref{prop_main}. Also $v = (1,1,2,1,3,1)$ and $v = (1,2,2,1,1,1)$ can be handled in this way (setting $U_3 = V$). Next $v = (1,1,2,1,1,1)$ implies by Claim 4 that $r_3$ and $r_4$ are both even, a contradiction. If $v = (1,1,2,1,3,3)$ then assume w.l.o.g. $3 \nmid r_1$. Then by Claim 4 we get that $r_3$ and $r_4$ are both even. To summarize, we are left with
$$v \in \{(2,2,2,1,1,1), (1,2,2,3,1,1), (1,2,2,3,3,1) \}.$$

If $v = (2,2,2,1,1,1)$ then we compute
$$|(U_3 \cap U_1)U_4| = |(U_3 \cap U_1)U_4| = \frac{|G|}{3}$$
	and $|(U_3\cap U_5)U_4| = \frac{|G|}{6}$. So $U_3U_4 \subseteq U_1U_4$, $U_3U_4 \subseteq U_5U_4$ and $U_3U_5 \subseteq U_1U_5$. Moreover
$$|U_1U_3| + |U_4U_3| + |U_5U_3| = \frac{2|G|}{3} + \frac{|G|}{6} + \frac{|G|}{6} = |G|.$$
So Lemma~\ref{lemma_main} (with $V = U_3$ and $U_1$ having the biggest index) implies that $(U_1,U_3,U_4,U_5)$ is not harmonic.

Next let $v = (1,2,2,3,1,1)$. Then $2 \mid r_3$ by Claim 4. Assume w.l.o.g. that $3 \nmid r_1$. Then $|(U_3 \cap U_1)U_4| = \frac{|G|}{2}$ implies $U_3U_4 \subseteq U_1U_4$ and $|(U_3 \cap U_1)U_2| = \frac{|G|}{3}$ implies $U_3U_2 \subseteq U_1U_2$. Observe
$$|U_3U_1| + |U_3U_2| + |U_3U_4| = \frac{|G|}{3} + \frac{|G|}{3} + \frac{|G|}{2} > |G|.$$
So $3 \mid r_2$, since otherwise $U_3U_4 \subseteq U_2U_4$ and we could apply Lemma~\ref{lemma_main}. But considering $U_3$, $U_4$ and $U_5$ we get from Lemma~\ref{lemma_ThreeTupleCase}a) that also $3 \mid r_5$, contradicting the assumption that $r_4$ and $r_5$ are coprime.

So $v =  (1,2,2,3,1,1)$. Then as in the previous case we obtain, w.l.o.g., that $2 \mid r_3$ and $3 \mid r_2$.
\end{proof}

We are not aware of any group satisfying the subgroup configuration described in Lemma~\ref{lemma_5TupleCase}. Such a group would provide a counterexample to Question 1.
In the case relevant for us we can in fact assume that a subgroup of index $3$ is involved in the coset partition. This is handled by the following lemma.

\begin{proposition}\label{prop_5Tupler11}
Let $r_2$,..., $r_5$ be pairwise coprime integers and assume also that $r_2$ is odd.
Then, for any group $G$, the $5$-tuple $(3,3r_2,6r_3,6r_4,6r_5)$ is not $G$-harmonic.
\end{proposition}
\begin{proof}
Assume to the contrary that there exist subgroups $U_1$, $U_2$, $U_3$, $U_4$, $U_5$ of $G$, such that $(U_1,U_2,U_3,U_4,U_5)$ is harmonic where
$[G:U_1] = 3$, $[G:U_2] = 3r_2$ and $[G:U_j] = 6r_j$ for $3 \leq j \leq 5.$
Then we can assume that $\alpha(x,y)$, and hence $|U_xU_y|$, is given by Lemma~\ref{lemma_5TupleCase} for any $1 \leq x,y \leq 5$, that $3 \mid r_2$ and $2 \mid r_3$. Since $G$ acts on the left cosets $G/U_1$ by left multiplication we obtain a homomorphism $\varphi: G \rightarrow S_3$, where $S_3$ denotes the symmetric group of degree $3$. We will use the bar notation for images of $\varphi$. Set $N = \ker(\varphi)$. Note that $N = U_1 \cap U_1^g \cap U_1^h$ where $\{1,g,h\}$ is a left transversal of $U_1$ in $G$. So $N \subseteq U_1$.

First observe that $\overline{U_1} \neq 1$ since otherwise $U_1 \subseteq N$ and $U_1U_2 \subseteq NU_2$, which is a subgroup of $G$. This contradicts $|U_1U_2| = \frac{2|G|}{3}$. Since the action of $U_1$ on $G/U_1$ admits a fix point, we deduce $|\overline{U_1}| = 2$. Next $\overline{U_2} \neq 1$ since otherwise $U_2U_1 \subseteq NU_1$, again contradicting $|U_1U_2| = \frac{2|G|}{3}$. If $3 \mid |\overline{U_2}|$ then $\overline{U_1}\overline{U_2} = \overline{G}$. Then using $N \subseteq U_1$ we get $U_1U_2 = NU_1U_2 = G$. Hence $|\overline{U_2}| = 2$. Since also $|U_1U_4| = \frac{2|G|}{3}$ we can deduce the same way that $|\overline{U_4}| = 2$.

Now consider the image of $U_3$. If $\overline{U_3} = 1$ then $U_3 \subseteq N$ and
$$\frac{|G|}{3} = |NU_4| = |NU_3U_4| \geq |U_3U_4| = \frac{|G|}{2}.$$
If $|\overline{U_3}|$ is divisible by $3$ then $\overline{U_1}\overline{U_3} = \overline{G}$, so $|G| = |NU_1U_3| = |U_1U_3| = \frac{|G|}{3}$. Hence $|\overline{U_3}| = 2$. So $[U_3:U_3 \cap N] = 2$ and since $[U_3:U_3 \cap U_2] = r_2$ is odd, being coprime with $r_3$, we conclude $U_2 \cap U_3 \leq N \cap U_3$. So on one hand $[G: U_2 \cap U_3] = 6r_2r_3$, while also
$$[G: U_2 \cap U_3] = [G:U_3 \cap N][U_3 \cap N:U_3 \cap U_2] = 12r_3[U_3 \cap N:U_3 \cap U_2].$$
This implies $[U_3 \cap N:U_3 \cap U_2] = \frac{r_2}{2}$ which is not an integer.
\end{proof}

Finally this allows us to prove Theorem A, though using elementary computer calculations.
\begin{center}
\textbf{Proof of Theorem A}
\end{center}

Let $G$ be a group of order smaller than $1440$ and let $\{g_iU_i\}_{i=1}^n$ be a partition of $G$ into cosets without multiplicity. Set $a_i = [G:U_i]$ for $1 \leq i \leq n$. Then by Lemma~\ref{lemma_restrindices} and Lemma~\ref{lemma_SmallGroupsInds} we know that $a_1$,...,$a_n$ are strictly bigger than $2$ and $\gcd(a_1,...,a_n)$ is divisible by $2$ or $3$.

Assume $\gcd(a_1,...,a_n)$ is divisible by $2$. Then using elementary computer calculations to obtain Egyptian fractions \eqref{eq_EgyFr} satisfying Lemma~\ref{lemma_restrindices}, where  $a_1$,...,$a_n$ are divisors of $|G|$ we get
$$|G| \in \{240, 360, 480, 720, 840, 960, 1008, 1080, 1200, 1320, 1344 \}.$$
If $|G| \neq 720$ this implies that $\{4,6,10\}$ or $\{4,6,14\}$ is a subset of $\{a_1,...,a_n\}$ and so $(U_1,...,U_n)$ is not harmonic by Proposition~\ref{prop_gcd2}. If $|G| = 720$ then $\{4,6,10\}$ or $\{6,4,20,8 \}$ is a subset of $\{a_1,...,a_n\}$ and we can conclude by Proposition~\ref{prop_gcd2} and Proposition~\ref{prop_4Tuple2Case} that $(U_1,...,U_n)$ is not harmonic.

Finally assume the $\gcd(a_1,...,a_n)$ is divisible by $3$. Then, again by elementary computer calculations, divisors  $a_1$,...,$a_n$ of $|G|$ can satisfy Lemma~\ref{lemma_restrindices} only if
$$|G| \in \{360, 540, 720, 1080, 1260 \}. $$
If $|G| \neq 1080$ then $\{3,6,9,15\}$ is a subset of $\{a_1,...,a_n\}$ and so $(U_1,...,U_n)$ is not harmonic by Proposition~\ref{prop_4Tuple3Case}. If $|G| = 1080$ then $\{3,6,9,15\}$ or $\{3,6,9,12,30\}$ is a subset of $\{a_1,...,a_n\}$ and so $(U_1,...,U_n)$ is not harmonic by Proposition~\ref{prop_4Tuple3Case} and Proposition~\ref{prop_5Tupler11}. \hfill \qed

\bibliographystyle{amsalpha}
\bibliography{HS}

\end{document}